\long\def\symbolfootnote[#1]#2{\begingroup\def\thefootnote{\fnsymbol{footnote}}\footnote[#1]{#2}\endgroup}

\documentclass{amsart}

\usepackage{hyperref}
\usepackage{indentfirst}
\usepackage{amssymb}
\usepackage{amsmath}
\usepackage{amsfonts}

\newtheorem{theorem}{Theorem}[section]
\newtheorem{corollary}[theorem]{Corollary}
\newtheorem{lemma}[theorem]{Lemma}
\theoremstyle{remark}
\newtheorem{remark}[theorem]{Remark}
\theoremstyle{definition}

\theoremstyle{proposition}

\numberwithin{equation}{section}

\begin{document}
\author{Linfeng Zhou}
\title{The spherically symmetric Finsler metrics with isotropic S-curvature}
\date{}
\maketitle

\begin{abstract}
In this paper, we investigate the spherically symmetric Finsler metrics with isotropic S-curvature and obtain a characterized equation.  As an application, we prove that these metrics with Douglas type must be Randers metrics or Berwald metrics. This result leads to two classification theorems. 
\\


\end{abstract}
\symbolfootnote[0]{This work is supported by an NSFC grant No. 11301187.}
\section{Introduction}
The S-curvature is a non-Riemannian quantity and plays an important role both in local and global Finsler geometry. It first appeared in Z. Shen's famous paper \cite{Sh} where he used it to derive a volume comparison theorem in Riemann-Finsler geometry. In 2007, B. Y. Wu and Y. L. Xin continued the investigations along the direction of generalizing the comparison techniques in Riemannain geometry to Finsler setting and obtained Hessian comparison theorem, Laplacian comparison theorem and volume comparison theorem under assumption of S-curvature and the flag curvature or Ricci curvature \cite{WX}.  Recently, S. Ohta and K. Sturm introduced the weighted Ricci curvature which also involves S-curvature and gave Bochner-Weitzenb\"ock formula and Li-Yau estimates on Finsler manifolds \cite{OS}.
 
 Furthermore, S-curvature is related to the Riemann tensor in the flowing way
 \[S_{\cdot k|m}y^m-S_{|k}=-\frac{1}{3}\{2R^m_{\ k\cdot m}+R^m_{\ m\cdot k}\},\]
 where $( \dots)_{|}$ is the horizontal covariant derivative and $(\dots )_{\cdot}$ is the vertical covariant derivative with respect to Chern connection \cite{Mo}. This identity leads to a well known result which is proved by X.  Chen , X. Mo and Z. Shen \cite{CMS}: if a $n$-dimensional Finsler metric $F$ is of scalar flag curvature and has almost isotropic S-curvature, i.e.
 \[S=(n+1){cF+\eta},\]
where $c$ is a scalar function and $\eta$ is a closed $1$-form, then the flag curvature, which is a generalization of the sectional curvature in Riemannian case, has following form
\[K=3\frac{c_{x^m}y^m}{F}+\sigma,\]
where $\sigma$ is a scalar function. Hence a natural problem is to study and characterize Finsler metrics of (almost) isotropic S-curvature. 

For a Randers metric $F=\alpha+\beta$,  X. Cheng and Z. Shen first calculated its S-curvature and proved that the condition of isotropic S-curvature is equivalent to the condition of isotropic E-curvature in 2003 \cite{CS1}.  By the viewpoint of Zermelo navigation, H. Xing found that  the isotropic S-curvature means its navigation vector field is conformal \cite{X}. This shows the condition of isotropic S-curvature actually is rather relaxed.  If imposing of scalar flag curvature and isotropic S-curvature, X. Cheng and Z. Shen completely determined the local structure of Randers metrics \cite{CS2}.  

For more general $(\alpha, \beta)$-metrics $F=\alpha\phi(\frac{\beta}{\alpha})$,  X. Cheng and Z. Shen obtained the characterized equations of $(\alpha,\beta)$-metrics with isotropic S-curvature \cite{CS}.  Using their result, Y. Shen and H. Tian gave an equivalent condition that an $(\alpha,\beta)$ metric admits
a measure $\mu$ with vanishing S-curvature and the characterization of $(\alpha,\beta)$-metrics admitting a measure $\mu$ with
weak isotropic S-curvature \cite{ST}.

Spherically symmetric metrics in Finsler setting are first introduced by S. F. Rutz who studied the spherically symmetric Finsler metrics in 4-dimensional time-space and generalized the classic Birkhoff theorem in general relativity to the Finsler case \cite{R}. The author obtained the general form of this kind of metrics to be $F=u\phi(r,s)$ in higher dimension, where $u=|y|$, $r=|x|$ and $s=\frac{\langle x, y\rangle}{|y|}$ and found many known examples belong to this type such as Riemannian space form, Funk metric, Berwald's example, Bryant's example \cite{Zh1}. Moreover, they have a nice symmetry and are invariant under any rotations. So it is quite valuable to study spherically symmetric Finsler metrics. Recently, some works have been carried out on this type of metrics \cite{HM}\cite{MST}\cite{MZ}\cite{Zh2}\cite{Zh3}. 

Since a spherically symmetric Finsler metric does not need to be $(\alpha, \beta)$ metric,  one problem is how to characterize this type of metrics with isotropic S-curvature.  

The main aim of this paper is to understand the spherically symmetric Finsler metrics with isotropic S-curvature.  We first obtain the characterized equation of this type of metrics with isotropic S-curvature if attached to a spherically symmetric volume form. Here the volume form is required to be spherically symmetric is not so strict since both Busemann-Hausdorff volume form and Holmes-Thompson volume form satisfy this condition.  It is not easy to solve the characterized equation directly to completely determine the metrics.  

By analyzing the characterized equation, we find that if imposing the Douglas type condition, then the spherically symmetric Finsler metrics with isotropic S-curvature must be Randers metrics or Berwald metric.  This result makes it possible to obtain the following two classification theorems.

\begin{theorem} \label{ctbh}
Let $F$ be a spherically symmetric Finsler metric on an open set $\Omega$ in $R^n$ and $dV_{BH}=\sigma_{BH}(x)dx$ be a Busemann-Hausdorff volume form.  $F$ is a Douglas metric and has isotropic S-curvature with respect to $dV_{BH}$ if and only if one of the following holds:
\begin{enumerate}
\item $F$ is Riemannian; 
\item $F$ is a Berwald metric which can be formulated by $$F=u\psi(\frac{s^2}{g(r)+s^2\int 4rc_2(r)g(r)dr})e^{-\int (\frac{2}{r}-2r^3c_2(r))dr}s,$$ 
where $c_2(r)$ is a smooth function and $g(r)=e^{\int (\frac{2}{r}-4r^3c_2(r))    dr}$;
\item $F=\sqrt{f(r)|y|^2+g(r)\langle x,y\rangle^2}+h(r)\langle x,y \rangle$ where the smooth functions $f(r)$, $g(r)$, $h(r)$ satisfy  \[f(r)>0, \quad f(r)+r^2\big(g(r)-h(r)^2\big)>0\] and
\[r^2fh\frac{d g(r)}{dr}+\big(2rfh+r^2f'h-2r^2fh'\big)g(r)+2ff'h-2f^2h'-2rfh^2-r^2f'h^3=0.\]

\end{enumerate}

\end{theorem}

\begin{theorem} \label{ctht}
Let $F$ be a spherically symmetric Finsler metric on an open set $\Omega$ in $R^n$ and $dV_{HT}=\sigma_{HT}(x)dx$ be a Holmes-Thompson volume form.  $F$ is a Douglas metric and has isotropic S-curvature with respect to $dV_{HT}$ if and only if  one of the following holds:
\begin{enumerate}

\item $F$ is Riemannian; 

\item $F$ is a Berwald metric which can be formulated by $$F=u\psi(\frac{s^2}{g(r)+s^2\int 4rc_2(r)g(r)dr})e^{-\int (\frac{2}{r}-2r^3c_2(r))dr}s,$$ 
where $c_2(r)$ is a smooth function and $g(r)=e^{\int (\frac{2}{r}-4r^3c_2(r))    dr}$;

\item $F=\sqrt{\frac{c}{r^2}|y|^2+g(r)\langle x,y\rangle^2}+h(r)\langle x,y \rangle$ where the positive constant $c$ and the smooth functions $g(r)$, $h(r)$ satisfy
 \[g(r)-h(r)^2>-\frac{c}{r^4},\quad 2h'(\frac{c}{r^2}+r^2g)-h(r^2g'-\frac{4c}{r^3})=0.\]
\end{enumerate}
\end{theorem}

\bigskip

\noindent{\bf Acknowledgements.} The author would like to thank  Professor Qun He and Guangzu Chen for their helpful discussion and the valuable comments.

\section{The S-curvature of the spherically symmetric Finsler metrics}
The notion of the S-curvature can be defined for an arbitrary given volume form  $dV=\sigma(x) dx$ on a Finsler manifold. There are two volume forms which are frequently encountered in Finsler geometry.  One is so-called the Busemann-Hausdorff volume form, the other is the Holmes-Thompson volume form. 

Let $F=F(x,y)$ be a Finsler metric on an $n$-dimensional manifold $M$. The Busemann-Hausdorff volume form $dV_{BH}=\sigma_{BH}(x)dx$ is given by
\[\sigma_{BH}(x)=\frac{Vol(B^n)}{Vol\{(y^i)\in R^n|F(x,y)<1\}}.\]
The Holmes-Thompson volume form $dV_{HT}=\sigma_{HT}(x)dx$ is defined as following:
\[\sigma_{HT}(x)=\frac{1}{Vol(B^n)}\int_{F(x,y)<1} det(g_{ij}(x,y))dy.\] 
In above definitions, the notation $Vol(B^n)$ means the Eucleadian volume of a unite ball in $R^n$. 
 
 The distortion $\tau=\tau(x,y)$ on $TM$ with respect to a given volume form $dV=\sigma(x) dx$ is defined by 
 \[\tau(x,y)=\ln \frac{\sqrt{det(g_{ij}(x,y))}}{\sigma(x)}.\] The S-curvature is the rate of change of the distortion along geodesics. More precisely, it is defined by
 \[S(x,y)=\frac{d}{dt}\big[\tau(c(t),\dot{c}(t))\big]|_{t=0},\]
 where $c(t)$ is the geodesic with $c(0)=x$ and $\dot{c}(0)=y$. $S(x,y)$ is positively homogeneous of degree one on the tangent vector $y$.
 
If we denote the spray coefficients $G^i$ of the Finsler metric $F$, then the S-curvature with respect to the volume form $dV=\sigma(x)dx$ can be formulated by
\begin{equation}\label{sf}
S=\frac{\partial G^m}{\partial y^m}-y^m\frac{\partial}{\partial x^m}(\ln \sigma).
\end{equation}

Usually it is not easy to explicitly work out the S-curvatures of the Busemann-Hausdorff volume form and the Holmes-Thompson volume form for an arbitrary Finsler metric.  However, a class of Finsler metric which is called $(\alpha,\beta)$ metric is computable for nearly all  quantities in Finsler geometry.  It is defined by a Rieman metric $\alpha$ and an 1-from $\beta$ on a manifold $M$ and includes many famous examples: the Randers metric, the square metric, the Matsumoto metric and so on.    X. Cheng and Z. Shen have obtained the formula of the $(\alpha,\beta)$ metric's S-curvatures of the Busemann-Hausdorff volume and the Holmes-Thompson \cite{CS}. 

A spherically symmetric Finsler metric generally can be expressed by $F=u\phi(r,s)$ where $u=|y|$, $r=|x|$ and $s=\frac{\langle x,y \rangle}{|y|}$ on an open set $\Omega$ in $R^n$ \cite{Zh1}. The metric tensor is given by 
\begin{eqnarray*}g_{ij}&=&\phi(\phi-s\phi_s)\delta_{ij}+(\phi_s^2+\phi\phi_{ss})x^ix^j+[s^2\phi\phi_{ss}-s(\phi-s\phi_s)\phi_s]\frac{y^i}{u}\frac{y^j}{u}\\
&&+[(\phi-s\phi_s)\phi_s-s\phi\phi_{ss}](x^i\frac{y^j}{u}+x^j\frac{y^i}{u}).
\end{eqnarray*}
The spray coefficients are formulated in terms of $\phi$ by \cite{MZ}
\[G^i=uPy^i+u^2Qx^i\]
where 
$$P=-\frac{1}{\phi}\big(s\phi+(r^2-s^2)\phi_s\big)Q+\frac{1}{2r\phi}(s\phi_r+r\phi_s)$$
and 
$$Q=\frac{1}{2r}\frac{-\phi_r+s\phi_{rs}+r\phi_{ss}}{\phi-s\phi_s+(r^2-s^2)\phi_{ss}}.$$

For a spherically symmetric Finsler metric, since it is invariant under any rotations in $R^n$, it is naturally to assume the given volume form is also invariant  under any rotations in $R^n$ which means 
\[dV= \sigma(r)dx \]
where $r=|x|$. According to (\ref{sf}),  we can direct compute the S-curvature of a spherically symmetric Finsler metric with respect to a given spherically symmetric volume form. 

\begin{lemma} \label{sfsfm}
Let $F=u\phi(r,s)$ be a spherically symmetric Finsler metric on an open set $\Omega$ in $R^n$ and $dV=\sigma(r) dx$ be a given spherically symmetric volume form. Suppose  that its spray coefficients are given by $G^i=uPy^i+u^2Qx^i$. The S-curvature of the given volume form is given by
\[S=(n+1)uP+u(r^2-s^2)Q_s+2usQ+uf(r)s\]
where $f(r)=-\frac{\sigma^{\prime}(r)}{r\sigma(r)}$.
\end{lemma}

\begin{proof} 
Since the spray coefficients  $G^i=uPy^i+u^2Qx^i$, one can get
\begin{eqnarray*}
\frac{\partial G^m}{\partial y^m}&=&\frac{\partial }{\partial y^m}(uPy^m+u^2Qx^m)\\
&=&uP\delta^{m}_{\ m}+(\frac{P}{u}-\frac{s}{u}P_s)y^my^m+P_sx^my^m+uQ_sx^mx^m+(2Q-sQ_s)x^my^m\\
&=&uPn+(\frac{P}{u}-\frac{s}{u}P_s)u^2+P_sus+uQ_sr^2+(2Q-sQ_s)us\\
&=&(n+1)uP+u(r^2-s^2)Q_s+2usQ.
\end{eqnarray*}
Furthermore, we have
\begin{eqnarray*}
y^m\frac{\partial}{\partial x^m}(\ln \sigma(r))&=&y^m\frac{\sigma^{\prime}(r)}{\sigma(r)}\frac{x^m}{r}\\
&=&u\frac{\sigma^{\prime}(r)}{r\sigma(r)}s.
\end{eqnarray*} 

Plugging above two equations into S-curvature's formula (\ref{sf}) will obtain the  result. 
\end{proof}

Similar to the calculation of the Busemann-Hausdorff volume form and the Holmes-Thompson volume form of the $(\alpha,\beta)$ metric in X. Cheng and Z. Shen's paper, we have following lemma for a spherically symmetric Finsler metric:
\begin{lemma}Let $F=u\phi(r,s)$ be a spherically symmetric Finsler metric on an open set $\Omega$ in $R^n$. Denote $dV_{BH}=\sigma_{BH}(x)dx$ and $dV_{HT}=\sigma_{HT}(x)dx$. Then 
\[\sigma_{BH}(r)=\frac{\int_{0}^{\pi}\sin ^{n-2} t dt}{\int_{0}^{\pi} \frac{\sin^{n-2}t }{\phi(r, r\cos t)^n}dt},\quad \sigma_{HT}(r)=\frac{\int_{0}^{\pi} (\sin^{n-2}t)T(r, r\cos t)dt}{\int_{0}^\pi \sin^{n-2}t dt}\]
where $T(r,s):=\phi(\phi-s\phi_s)^{n-2}[(\phi-s\phi_s)+(r^2-s^2)\phi_{ss}]$.
\end{lemma} 
The proof is omitted here and one can consult \cite{CS}  for the details.  

\begin{remark} Obviously, both these two volume forms are spherically symmetric and the S-curvature can be computed by using Lemma \ref{sfsfm}.
\end{remark}

A Finsler metric $F$ on an $n$-dimensional manifold $M$ is said to have \emph{isotropic S-curvature} with respect to the given volume form $dV=\sigma(x)dx$ if there exits a function $c=c(x)$ on $M$ so that
\[S=(n+1)cF.\]
$F$ is said to have \emph{constant S-curvature} if $c=constant$.

By combining the definition of isotropic S-curvature and Lemma \ref{sfsfm}, the characterized equation of the spherically symmetric Finsler metrics with isotropic S-curvature can be easily derived. 
\begin{theorem} \label{tsf}
Let $F=u\phi(r,s)$ be a spherically symmetric Finsler metric on an open set $\Omega$ in $R^n$ and $dV=\sigma(r) dx$ be a given spherically symmetric volume form. Suppose  that its spray coefficients are given by $G^i=uPy^i+u^2Qx^i$. Then $F$ has isotropic S-curvature if and only if  the equation
\[(n+1)P+(r^2-s^2)Q_s+2sQ+f(r)s=(n+1)c\phi\]
holds where $f(r)=-\frac{\sigma^{\prime}(r)}{r\sigma(r)}$ and $c=c(r)$ is a scalar function.
\end{theorem}

\begin{proof}  Lemma \ref{sfsfm} and the definition of isotropic S-curvature imply 
\[(n+1)P+(r^2-s^2)Q_s+2sQ+f(r)s=(n+1)c(x)\phi\]
where $c(x)$ is a function on $\Omega$.  The spherical symmetry of the RHS in above equation provides $c(x)$ is a function of $r=|x|$. 
\end{proof}

\begin{remark} The characterized equation in above theorem actually is a 3-order partial differential equation of $\phi$ by plugging the formula $P$ and $Q$ into the LHS. 
\end{remark}

\section{The rigidity of the spherically symmetric Douglas metrics with isotropic S-curvature}

In previous section,  we know the characterized equation of the spherically symmetric Finsler metrics with isotropic S-curvature. A natural question emerges:  how to classify them? It is not easy to solve the equation directly since it is 3-order nonlinear pde.   However, under a certain assumption:  Douglas type, we can completely determine this kind of metrics.  

The Douglas metric can be regarded as the generalization of Berwald metric and was first introduced in \cite{BM}. A Finsler metric is called \emph{Douglas metric} if it's spray coefficients satisfy
\[G^i=\frac{1}{2}\Gamma^i_{\ jk}y^jy^k+P(x,y)y^i\]
in a local coordinate, which means it is locally projectively related to a Berwald metric. 

For a spherically symmetric Douglas metric,  Mo, Sol—rzano and Tenenblat  yield the following equivalent condition on the spray coefficients: 
\begin{lemma} \cite{MST}
\label{ldm}
Let $F=u\phi(r,s)$ be a spherically symmetric Finsler metric on an open set $\Omega$ in $R^n$. Suppose  that its spray coefficients are given by $G^i=uPy^i+u^2Qx^i$. Then $F$ is a Douglas metric if and only if 
$Q=c_1(r)+c_2(r)s^2$. 
\end{lemma}

They calculated the Douglas tensor and utilized it vanishing to draw above conclusion. Actually, by directly observing the $G^i=uPy^i+u^2Qx^i$ and applying it projectively related to a Berwald metric, it can be deduced that $u^2Q$ must the quadratic form of the tangent vector $y$. This leads to $Q$ to be the form of $c_1(r) + c_2(r)s^2$.  

If we restrict to a spherically symmetric Douglas metric with isotropic S-curvature, it has a rigidity property. More precisely, the following theorem holds: 

\begin{theorem} \label{mt}
 Let $F$ be a spherically symmetric Finsler metric on an open set $\Omega$ in $R^n$ and  $dV=\sigma(r) dx$ be a given spherically symmetric volume form.  If $F$ is a Douglas metric and has isotropic S-curvature with respect to the volume form $dV$, then either 
 \begin{enumerate}
 \item $F$ is a Randers metric (Riemann metric included)
  or 
\item $F$ is a Berwald metric which can be formulated by $$F=u\psi(\frac{s^2}{g(r)+s^2\int 4rc_2(r)g(r)dr})e^{-\int (\frac{2}{r}-2r^3c_2(r))dr}s,$$ 
where $c_2(r)$ is a smooth function and $g(r)=e^{\int (\frac{2}{r}-4r^3c_2(r))    dr}$.
\end{enumerate}
\end{theorem}

Before proving Theorem \ref{mt},  we need a Lemma. 

\begin{lemma}  \label{So}
Suppose there is the following 1-order non-homogenous  linear partial differential equation   
\[\frac{\partial \phi(r,s)}{\partial r}+\big(\frac{s}{r}-2rc_2(r)(r^2-s^2)s\big)\frac{\partial \phi(r,s)}{\partial s}=-(\frac{1}{r}-2rc_2(r)s^2)\phi.\]
Then its general solution is given by
\[\phi=\psi(\frac{s^2}{g(r)+s^2\int 4rc_2(r)g(r)dr})e^{-\int (\frac{2}{r}-2r^3c_2(r))dr}s,\]
where $g(r)=e^{\int (\frac{2}{r}-4r^3c_2(r))    dr}$.
\end{lemma}

\begin{proof}
It's characteristic equation is 
\[\frac{dr}{1}=\frac{ds}{\frac{s}{r}-2rc_2(r)(r^2-s^2)s}=\frac{d\phi}{-(\frac{1}{r}-2rc_2(r)s^2)\phi}.\]
From $\frac{dr}{1}=\frac{ds}{\frac{s}{r}-2rc_2(r)(r^2-s^2)s}$, we know that
\[\frac{d}{dr}(s^2)=(\frac{2}{r}-4c_2r^3)s^2+4c_2rs^4.\]
This is a Bernoulli equation and can be rewritten as
\[\frac{d}{dr}(\frac{1}{s^2})=(-\frac{2}{r}+4c_2r^3)\frac{1}{s^2}-4c_2r.\]
Thus above equation turns into a linear 1-order ODE of $\frac{1}{s^2}$.  One can easily get its solution
\[\frac{1}{s^2}=e^{\int (-\frac{2}{r}+4c_2r^3)dr}(c-\int 4c_2re^{\int (\frac{2}{r}-4c_2r^3)dr}  dr).\]
Therefore, one first integral of the original equation can be chosen to be
 \[\frac{s^2}{e^{\int (\frac{2}{r}-4c_2r^3)dr}+s^2\int 4c_2re^{\int (\frac{2}{r}-4c_2r^3)dr}dr}=\frac{1}{c}.\]
 In order to get another independent first integral, from the characteristic equation, we notice that
 \[\frac{d \ln s}{-\frac{1}{r}+2c_2r^3-2c_2rs^2}=\frac{dr}{-1}=\frac{d \ln \phi}{\frac{1}{r}-2c_2rs^2}.\]
 It implies that
 \[\frac{d \ln s - d \ln \phi}{-\frac{2}{r}+2c_2r^3}=\frac{dr}{-1}.\]
 Integrating above equation yields
 \[\ln \frac{s}{\phi}-\int(\frac{2}{r}-2c_2r^3)dr=c.\]
 Obviously, it can be selected as another independent first integral. Hence the general solution of the original equation can be expressed by 
   \[\Psi\big(\frac{s^2}{e^{\int (\frac{2}{r}-4c_2r^3)dr}+s^2\int 4c_2re^{\int (\frac{2}{r}-4c_2r^3)dr}dr},\ln \frac{s}{\phi}-\int(\frac{2}{r}-2c_2r^3)dr\big)=0.\]
From above equality, one may solve that
\[\phi=\psi(\frac{s^2}{g(r)+s^2\int 4rc_2(r)g(r)dr})e^{-\int ( \frac{2}{r}-2r^3c_2(r)) dr}s,\]
where $g(r)=e^{\int (\frac{2}{r}-4r^3c_2(r))    dr}$.
\end{proof}

\begin{proof} [Proof of Theorem \ref{mt}]Lemma \ref{ldm} tells that the spray coefficients $G^i=uPy^i+u^2Qx^i$ of $F=u\phi(r,s)$ must satisfy
\[Q=c_1(r)+c_2(r)s^2.\]
Meanwhile, by Theorem \ref{tsf}  the isotropy of the S-curvature implies 
\[(n+1)P+(r^2-s^2)Q_s+2sQ+f(r)s=(n+1)c(r)\phi.\]
Therefore, $P$ can be written as 
\[P=c(r)\phi+b(r)s\]
where $b(r):=-\frac{1}{n+1}[2c_1(r)+2c_2(r)r^2+f(r)]$.

Now we are ready to use the technique developed in \cite{MZ} to determine the metric $F$.  Since 
$$P=-\frac{1}{\phi}\big(s\phi+(r^2-s^2)\phi_s\big)Q+\frac{1}{2r\phi}(s\phi_r+r\phi_s)$$
and 
$$Q=\frac{1}{2r}\frac{-\phi_r+s\phi_{rs}+r\phi_{ss}}{\phi-s\phi_s+(r^2-s^2)\phi_{ss}},$$
one gets a equation system
\begin{equation*}
\left\{ \begin{array}{l}
-\frac{1}{\phi}\big(s\phi+(r^2-s^2)\phi_s\big)(c_1(r)+c_2(r)s^2)+\frac{1}{2r\phi}(s\phi_r+r\phi_s)=c(r)\phi+b(r)s \\
 \frac{1}{2r}\frac{-\phi_r+s\phi_{rs}+r\phi_{ss}}{\phi-s\phi_s+(r^2-s^2)\phi_{ss}}=c_1(r)+c_2(r)s^2.   
 \end{array} \right.
          \end{equation*}
Rewriting the system will provide
\begin{equation} \label{eq1}
\left\{ \begin{array}{l}
r[2(r^2-s^2)(c_1+c_2s^2)-1]\phi_s-s\phi_r+2[brs+rs(c_1+c_2s^2)]\phi+2rc\phi^2=0\\
r[2(r^2-s^2)(c_1+c_2s^2)-1]\phi_{ss}-s\phi_{rs}+\phi_r+2r(c_1+c_2s^2)(\phi-s\phi_s)=0.
 \end{array} \right.
          \end{equation}         
Differentiating the first equation with respect to variable $s$ and subtracting the second equation, we obtain
\[  [brs+2c_2r(r^2-s^2)s]\phi_s-\phi_r+[br+2c_2rs^2]\phi+2cr\phi\phi_s=0. \]
By combining this equation and the first equation in (\ref{eq1}),  it can be deduced that
\[r[2c_1(r^2-s^2)-1-bs^2]\phi_s-2crs\phi\phi_s+(b+2c_1)rs\phi+2cr\phi^2=0.\]
Actually, above equation can be written as
\begin{equation}\label{eq2}
\big[r\big(2c_1(r^2-s^2)-1-bs^2\big)\frac{1}{\phi^2}\big]_s- (4crs\frac{1}{\phi})_s=0.
\end{equation}

If $r\big(2c_1(r^2-s^2)-1-bs^2\big)\neq 0$, integrating the identity will solve $\phi$ to obtain that $F=u\phi$ is a Randers metric.

If $r\big(2c_1(r^2-s^2)-1-bs^2\big)= 0$ and $c\neq 0$, the equation (\ref{eq2}) concludes that $\phi$ is a singular Kropina metric. Here we omit this case since the Finsler metric is assumed to be regular. 

If $r\big(2c_1(r^2-s^2)-1-bs^2\big)= 0$ and $c=0$, which means $c_1=\frac{1}{2r^2}$ and $b=-\frac{1}{r^2}$, above calculation indicates the first equation in (\ref{eq1}) implies the second equation. In this case, the first equation turns to be 
\[\big(-\frac{s}{r}+2c_2r(r^2-s^2)s\big)\phi_s-\phi_r+\big(-\frac{1}{r}+2c_2rs^2\big)\phi=0.\]
According to Lemma \ref{So}, the general solution of the above equation is given by 
$$\phi=\psi(\frac{s^2}{g(r)+s^2\int 4rc_2(r)g(r)dr})e^{-\int (\frac{2}{r}-2r^3c_2(r))dr}s,$$ 
where $g(r)=e^{\int (\frac{2}{r}-4r^3c_2(r)) dr}$.
\end{proof}

From the proof of Theorem \ref{mt}, one can have the following corollary.

\begin{corollary} Under the same assumption of Theorem \ref{mt}, if $F$ is a Douglas metric and $S=0$ with respect to the volume form $dV$, then $F$ must be Berwaldian. 
\end{corollary}


\section{The classification theorems}
Theorem \ref{mt} concludes a spherically symmetric Douglas metric with isotropic S-curvature must be either a Randers metric or a Berwald metric if the given volume form is also spherically symmetric. The spherically symmetric Berwald metrics have already been completely classified in \cite{MZ}. The result makes us focusing on the Randers metric with isotropic S-curvature. 

 Let $F=\alpha+\beta$ be a Randers metric on a manifold $M$ with $n$-dimension, where $\alpha=\sqrt{a_{ij}y^iy^j}$ is a Riemann metric and $\beta=b_iy^i$ is an 1-from.  It is easy to verify that $F$ is a Finsler metric if and only if $||\beta_{x}||_{\alpha}<1$ for all $x\in M$. Let
 \[r_{ij}:=\frac{1}{2}(b_{i;j}+b_{j;i}), s_{ij}:=\frac{1}{2}(b_{i;j}-b_{j;i}),\]
 \[s^i_{\ j}:=a^{ik}s_{kj}, s_j:=b_is^i_{\ j}=b^ks_{kj}, e_{ij}=r_{ij}+b_is_j+b_js_i,\]
where $";"$ denotes the covariant derivative with respect to the Levi-Civita connection
of $\alpha$. The spray coefficients $G^i$ of $F$ and the spray coefficients $G^i_{\ \alpha}$ of $\alpha$ are related by 
\[G^i=G^i_{\ \alpha}+Py^i+Q^i,\]
where
 \[P:=\frac{e_{00}}{F}-s_0, Q^i=\alpha s^i_{\ 0},\] 
where $e_{00}:=e_{ij}y^iy^j$, $s_0:=s_iy^i$ and $s^i_{\ 0}:=s^i_{\ j}y^j$.

The following lemma is first proved by S. B\'acs\'o and M. Matsumoto \cite{BM}:
\begin{lemma}\label{drc}
For a Randers metric $F=\alpha+\beta$ on a manifold, it is  a Douglas metric if and only if $\beta$ is closed.   
\end{lemma}

For a Randers metric $F=\alpha+\beta$, if let \[\rho:=\ln \sqrt{1-||\beta_{x}||_{\alpha}^2},\]
the Busemann-Hausdorff volume form $dV_{BH}$ and the Riemannian volume form $dV_{\alpha}$ are related by 
\[dV_{BH}=e^{(n+1)\rho(x)}dV_{\alpha}.\]
Plugging above identities into the formula of the S-curvature (\ref{sf}) gives the S-curvature of Randers metric
\[S=(n+1){\frac{e_{00}}{2F}-(s_0+\rho_0)}\]
where $\rho_0:=y^k\rho_{x^k}(x)$.

\begin{lemma}\cite{CS1} \label{src}
 Let $F=\alpha+\beta$  be a Randers metric on a manifold $M$ with $n$-dimension and  the given volume form be the Busemann-Hausdorff volume form $dV_{BH}$.  The following are equivalent
\begin{enumerate}
\item $S=(n+1)cF$,
\item $e_{00}=2c(\alpha^2-\beta^2)$,
\end{enumerate}
where $c=c(x)$ is a scalar function on $M$.
\end{lemma}

For a Randers metric $F=\alpha+\beta$,  by a direct computation, it's Holmes-Thompson volume form $dV_{HT}$ is given by 
\[dV_{HT}=dV_{\alpha}.\]
According to the formula of S-curvature (\ref{sf}), we have 
\[S=(n+1){\frac{e_{00}}{2F}-s_0}.\]

\begin{lemma} \label{src1}
 Let $F=\alpha+\beta$  be a Randers metric on a manifold $M$ with $n$-dimension and  the given volume form be the Holmes-Thompson volume form $dV_{HT}$. 
 If $F$ has isotropic S-curvature, then $\beta$ is parallel 1-form to $\alpha$.
 \end{lemma}
\begin{proof}  $F$ has isotropic S-curvature i.e.
\[S=(n+1)cF.\]
On the other hand, a Randers metric's S-curvature is given by
\[S=(n+1){\frac{e_{00}}{2F}-s_0}.\]
This implies that
\[ {\frac{e_{00}}{2F}-s_0}=cF.\] 
Plugging $F=\alpha+\beta$ into above identity will give
\[e_{00}=2c\alpha^2+2c\beta^2+2s_0\beta+\alpha(2s_0+4c\beta).\]
Notice $e_{00}$, $2c\alpha^2+2c\beta^2+2s_0\beta$ is the rational part and $ \alpha(2s_0+4c\beta)$ is the irrational part, thus 
\[e_{00}=2c\alpha^2+2c\beta^2+2s_0\beta, \quad 2s_0+4c\beta=0.\]
It means $$s_i=-2cb_i.$$
Contracting with $b^i$ implies that $c=0$. Therefore \[s_i=2cb_i=0.\]
So
\[e_{00}=r_{ij}y^iy^j+s_ib_j+s_jb_i=r_{ij}y^iy^j=2c\alpha^2+2c\beta^2+2s_0\beta=0.\]
It means  $\beta$ is parallel 1-form to $\alpha$.
\end{proof}

For a spherically symmetric Douglas metric with isotropic S-curvature,  we have a complete classification, i.e. Theorem \ref{ctbh} and Theorem \ref{ctht}, for the two  specific form volume: Busemann-Hausdorff volume form and the Holmes-Thompson volume form.

\begin{proof} [Proof of Theorem \ref{ctbh}]
Let us first prove the necessity. By Theorem \ref{mt}, $F=u\phi(r,s)$ must be a Randers metric (Riemann metric included) or a Berwald metric which can be written as $$F=u\psi(\frac{s^2}{g(r)+s^2\int 4rc_2(r)g(r)dr})e^{-\int (\frac{2}{r}-2r^3c_2(r))dr}s,$$ 
where $c_2(r)$ is a smooth function and $g(r)=e^{\int (\frac{2}{r}-4r^3c_2(r))    dr}$.  Thus we only needs to discuss the Randers case. 

If $F$ is a Randers metric, then $F$ has the following form
\[F=\sqrt{f(r)|y|^2+g(r)\langle x,y\rangle^2}+h(r)\langle x,y \rangle\]
where $f(r)$, $g(r)$ and $h(r)$ are smooth functions. In this case, it means 
\[\alpha=\sqrt{f(r)|y|^2+g(r)\langle x,y\rangle^2}, \quad \beta=h(r)\langle x,y \rangle.\]
Hence
\[a_{ij}=f(r)\delta_{ij}+g(r)x^ix^j, \quad b_i=h(r)x^i.\]
It is easy to verify that \[\frac{\partial b_i}{\partial x^j}-\frac{\partial b_j}{\partial x^i}=0.\]
 So $\beta$ is closed. 
Furthermore, we can yield
\[det(a_{ij})=\big(f+r^2g\big)f^{n-1},\quad  a^{ij}=\frac{1}{f}\delta^{ij}-\frac{g}{f(f+r^2g)}x^ix^j\]
and 
\[||\beta||_{\alpha}^2=b_ib_ja^{ij}=\frac{r^2h^2}{f+r^2g}.\]
Since $F=\alpha+\beta$ is a Randers metric, it implies that $\alpha$ is positive definite and $||\beta||_{\alpha}^2<1$. Therefore 
\[f(r)>0, \quad f(r)+r^2g(r)>0,\quad  f(r)+r^2(g(r)-h(r)^2)>0.\]
Obviously the third inequality implies the second inequality. 
The Christoffel symbols $\Gamma^k_{ij}(x)$ of $\alpha$ is given by
\begin{eqnarray*}
\Gamma^k_{ij}&=&\frac{a^{kl}}{2}\Big\{\frac{\partial a_{li}}{\partial x^j}+\frac{\partial a_{lj}}{\partial x^i}-\frac{\partial a_{ij}}{\partial x^l}\Big \}\\
&=&\frac{1}{2}\big( \frac{fg'-2f'g}{rf(f+r^2g)}x^ix^jx^k+\frac{2rg-f'}{r(f+r^2g)}x^k\delta_{ij}+\frac{f'}{rf}(x^i\delta_{kj}+x^j\delta_{ki})\big).
\end{eqnarray*} 
Thus
\begin{eqnarray*}
b_{i;j}&=&\frac{\partial b_i}{\partial x^j}-b_k\Gamma^k_{\ ij}\\
&=&\frac{1}{2}\frac{h(rf'+2f)}{f+r^2g}\delta_{ij}+(\frac{h'}{r}-\frac{1}{2}\frac{h(r^2g'+2f')}{r(f+r^2g)})x_ix_j.
\end{eqnarray*} 
For $\beta$ is closed, we have
\begin{eqnarray*}
e_{00}&=&r_{ij}y^iy^j=\frac{1}{2}(b_{i;j}+b_{j;i})y^iy^j\\
&=&\frac{1}{2}\frac{h(rf'+2f)}{f+r^2g}u^2+\Big(\frac{h'}{r}-\frac{1}{2}\frac{h(r^2g'+2f')}{r(f+r^2g)}\Big)s^2u^2.
\end{eqnarray*}
By Lemma \ref{src}, $e_{00}$ satisfies $e_{00}=2c(\alpha^2-\beta^2)$, i.e.
\[\frac{1}{2}\frac{h(rf'+2f)}{f+r^2g}u^2+\Big(\frac{h'}{r}-\frac{1}{2}\frac{h(r^2g'+2f')}{r(f+r^2g)}\Big)s^2u^2=2c\big(f+(g-h^2)s^2\big)u^2\]
where $c=c(r)$ is a function of $r$. 
Comparing LHS and RHS of the equation will obtain
\begin{equation*}
\left\{ \begin{array}{l}
 \frac{1}{2}\frac{h(rf'+2f)}{f+r^2g}=2cf\\
 \frac{h'}{r}-\frac{1}{2}\frac{h(r^2g'+2f')}{r(f+r^2g)}=2c(g-h^2).
  \end{array} \right.
          \end{equation*}
 Eliminating the function $c(r)$, we get 
 \[r^2fh\frac{d g(r)}{dr}+\big(2rfh+r^2f'h-2r^2fh'\big)g(r)+2ff'h-2f^2h'-2rfh^2-r^2f'h^3=0.\]

 Conversely, it is obvious that if $F$ is Riemannian or Berwaldian, then $F$ is Douglas metric and has a vanishing $S$-curvature. 
 
 If  $F=\sqrt{f(r)|y|^2+g(r)\langle x,y\rangle^2}+h(r)\langle x,y \rangle$ where $f$, $g$ and $h$ satisfy $f(r)>0$, $f(r)+(g(r)-h(r)^2)r^2>0$ is a spherically symmetric Randers metric and $\beta$ is closed by the previous calculation. By Lemma \ref{drc}, $F$ is a Douglas metric. In above computation, one can see that each step is reversible. The equation \[r^2fh\frac{d g(r)}{dr}+\big(2rfh+r^2f'h-2r^2fh'\big)g(r)+2ff'h-2f^2h'-2rfh^2-r^2f'h^3=0\] assures  $F$ has isotropic S-curvature with respect to $dV_{BH}$. 
 
\end{proof}

\begin{proof} [Proof of Theorem \ref{ctht}]
Firstly, if $F$ is a Douglas metric with isotropic S-curvature, by Theorem \ref{mt}, $F$ must be a Randers metric (Riemann metric included) or a Berwald metric which can be written as $$F=u\psi(\frac{s^2}{g(r)+s^2\int 4rc_2(r)g(r)dr})e^{-\int (\frac{2}{r}-2r^3c_2(r))dr}s,$$ 
where $c_2(r)$ is a smooth function and $g(r)=e^{\int (\frac{2}{r}-4r^3c_2(r))    dr}$. 

If $F$ is a Randers metric, it can be expressed  by 
\[F=\sqrt{f(r)|y|^2+g(r)\langle x,y\rangle^2}+h(r)\langle x,y \rangle.\]
By Lemma \ref{src1}, $\beta$ must be parallel 1-from which is equivalent to 
\[b_{i;j}=0.\]
Plugging 
\[b_{i;j}=\frac{1}{2}\frac{h(rf'+2f)}{f+r^2g}\delta_{ij}+(\frac{h'}{r}-\frac{1}{2}\frac{h(r^2g'+2f')}{r(f+r^2g)})x_ix_j
\] 
into above identity obtains
\begin{equation} \label{eq3}
\left\{ \begin{array}{l}
\frac{1}{2}\frac{h(rf'+2f)}{f+r^2g}=0\\
\frac{h'}{r}-\frac{1}{2}\frac{h(r^2g'+2f')}{r(f+r^2g)}=0.
 \end{array} \right.
          \end{equation}     
From the first equation, we know that $$h=0$$ or $$rf'+2f=0.$$
If $h=0$, it means $F$ to be Riemannian. If $rf'+2f=0$, solving it will arrive at
\[f=\frac{c}{r^2}.\]
Plugging $f$ into the second equation of (\ref{eq3}), one can get
\[2h'(\frac{c}{r^2}+r^2g)-h(r^2g'-\frac{4c}{r^3})=0.\]
The condition of $F$ to be a  Randers metric leads to $c$ is a positive constant and $g(r)-h(r)^2>-\frac{c}{r^4}$.

Conversely it is easy to verify if $F$ is Riemannian and Berwaldian, $F$ is a Douglas metric and the S-curvature of $F$ vanishes. 

If $F=\sqrt{\frac{c}{r^2}|y|^2+g(r)\langle x,y\rangle^2}+h(r)\langle x,y \rangle$ and $g(r)$, $h(r)$ satisfy
\[2h'(\frac{c}{r^2}+r^2g)-h(r^2g'-\frac{4c}{r^3})=0,\]
then $\beta$ is a parallel 1-form. Thus $F$ is also a Douglas metric and the S-curvature vanishes. 
\end{proof}

{\small DEPARTMENT OF MATHEMATICS, EAST CHINA NORMAL UNIVERSITY }

{\small SHANGHAI 200062, CHINA}\newline
{\small E-mail address: lfzhou@math.ecnu.edu.cn}
\end{document}